\documentclass[12pt,a4paper]{article}
\usepackage{bbm}
\usepackage{graphicx}
\graphicspath{ {./images/} }
\usepackage{wrapfig}
\usepackage[mathscr]{euscript}
\usepackage{subcaption}
\usepackage{hyperref}
\usepackage{epstopdf}
\usepackage{enumerate}
\usepackage{amsmath,amsfonts,amssymb,amsthm,epsfig,epstopdf,titling,url,array}
\usepackage{color}
\usepackage{framed}
\usepackage[utf8]{inputenc}
\usepackage[english]{babel}

\usepackage{xparse}
\usepackage{multirow}
\usepackage{longtable}
\NewDocumentCommand{\INTERVALINNARDS}{ m m }{
	#1 {,} #2
}
\NewDocumentCommand{\interval}{ s m >{\SplitArgument{1}{,}}m m o }
{
	\IfBooleanTF{#1}{
		\left#2 \INTERVALINNARDS #3 \right#4
	}{
		\IfValueTF{#5}{
			#5{#2} \INTERVALINNARDS #3 #5{#4}
		}{
			#2 \INTERVALINNARDS #3 #4
		}
	}
}
\newtheorem{theorem}{Theorem}[section]
\newtheorem{corollary}{Corollary}[theorem]

\newtheorem{lemma}[theorem]{Lemma}

\newtheorem{obs}{Observation}
\newtheorem{Remark}{Remark}[section]

\usepackage{authblk}

\begin{document} 
	\title{\textbf{On dynamics of the Chebyshev's method for quartic polynomials}}
	\author[1]{Tarakanta Nayak\footnote{Corresponding author, tnayak@iitbbs.ac.in}}
	\author[1]{Soumen Pal\footnote{sp58@iitbbs.ac.in}}
	\affil[1]{School of Basic Sciences, 
		Indian Institute of Technology Bhubaneswar, India}
	\date{}
	\maketitle
	\begin{abstract}
		Let $p$ be a normalized (monic and centered) quartic polynomial  with non-trivial symmetry groups. It is already known that if $p$ is unicritical, with only two distinct roots with the same multiplicity or having a root at the origin then the Julia set of its Chebyshev's method $C_p$ is connected and  symmetry groups of $p$ and $C_p$ coincide~[Nayak, T., and Pal, S., Symmetry and dynamics of Chebyshev's method, \cite{Sym-and-dyn}]. Every other quartic polynomial is shown to be of the form $p_a (z)=(z^2 -1)(z^2-a)$ where $a \in \mathbb{C}\setminus \{-1,0,1\}$. Some dynamical aspects of the Chebyshev's method $C_a$ of $p_a$ are investigated in this article for all real $a$. It is  proved that all the extraneous fixed points of $C _a$ are repelling which gives that there is no invariant Siegel disk for $C_a$. It is also shown that there is no Herman ring in the Fatou set of $C_a$. For positive $a$, it is proved that at least two immediate basins of $C_a$ corresponding to the roots of $p_a$ are unbounded and simply connected. For negative $a$, it is however proved that all the four immediate basins of $C_a$ corresponding to the roots of $p_a$ are unbounded and those corresponding to $\pm i\sqrt{|a|}$ are simply connected. 
	\end{abstract}
%=====================================================================
\section{Introduction}
A root-finding method is a function from the space of all polynomials that assigns a rational map $F_p$ to a polynomial $p$  such that each root of $p$ is an attracting fixed point of $F_p$, i.e., if $z_0$ is a root of $p$ then $F_p (z_0)=z_0$ and $|F_p'(z_0)|<1$. Though there are several such methods appearing in the literature, the family of K\"{o}nig's methods~\cite{BH2003} and Chebyshev-Halley methods~\cite{CCV} seem to be  comparatively well-studied among them. The Newton method $N_p$ is the first member of the  K\"{o}nig's methods and its order of convergence (i.e., the local degree of $N_p$ at each of the simple roots of $p$) is two. Further, it has no finite extraneous fixed point, i.e., each finite fixed point of $N_p$ is a root of $p$.  Note that the sequence  of forward iterates of every root-finding method converges to a root of a polynomial in a suitably small neighborhood of the root.  The non-existence of finite extraneous fixed points for the Newton method have been found to be crucial in the study of its global dynamics (i.e., not only in a neighborhood of the roots of the polynomial but in $\widehat{\mathbb{C}}$) of the Newton method. For example, this is precisely the reason why the Julia set (it is the set of all points in $\widehat{\mathbb{C}}$ at which the sequence of iterates $\{N_p ^{n}\}_{n >0}$ is not normal~\cite{Beardon_book}) of the Newton method applied to a polynomial is connected \cite{MS}. These are possibly some reasons for which  this method has drawn a good amount of attention of researchers. However there are root-finding methods whose order of convergence is three and which has finite extraneous fixed points. One such, namely the Chebyshev's method is the subject of this article.
\par
For a polynomial $p$, its Chebyshev's method is given by $$ C_{p}(z)=z-\left[1+\frac{1}{2}L_{p}(z)\right]\frac{p(z)}{p'(z)},
 $$
where $ L_{p}(z)=\frac{p(z)p''(z)}{[p'(z)]^{2}}.$

\par 
The Fatou set of a rational map $R$, denoted by $\mathcal{F}(R)$ is the set of all points in $\widehat{\mathbb{C}}$ in a neighborhood of which $\{R^n\}_{n>0}$ is normal. Its complement in $\widehat{\mathbb{C}}$ is known as the Julia set of $R$ and is denoted by $\mathcal{J}(R)$. A maximally connected open subset $U$ of the Fatou set, called a Fatou component is said to be $p-$periodic if $R^p (U)=U$. It is well-known that for every Fatou component $U$ of a rational map  $R$, there is a $k$ such that $R^{k}(U)$ is periodic. A periodic Fatou component can be an attracting domain, a parabolic domain, a Siegel disk or a Herman ring. Other properties of these Fatou component can be found in \cite{Beardon_book}. We are concerned with the dynamics (the Fatou and the Julia sets) of the Chebyshev's method applied to polynomials. The nature of extraneous fixed points and some other  dynamical aspects of Chebyshev's method applied to quadratic and cubic polynomial have been studied in \cite{Olivo2015}. Existence of superattracting cycles for Chebyshev's method applied to cubic polynomial  are investigated in \cite{GV2020}. The first systematic study of the dynamics of  Chebyshev's method applied to cubic polynomials can be found in \cite{Nayak-Pal2022}. The family of Chebyshev's method applied to cubic polynomials is parametrized in terms of the multiplier of an extraneous fixed point and its dynamics is determined for parameters in $[-1,1]$. A discussion of the Chebyshev's method applied to some  quartic polynomials appeared in \cite{Sym-and-dyn} which mainly deals with the relation between the  group  of Euclidean isometries preserving the Julia set  of a polynomial and its Chebyshev's method, whenever the earlier is non-trivial.   It is shown that both these groups are isomorphic for all centered cubic and quartic polynomials with non-trivial symmetry groups.

This article deals with the quartic polynomials. First, we parametrize the family of maps arising as the Chebyshev's method of quartic polynomials. 
  \par   A root-finding method $F_p$ is said to satisfy the Scaling theorem (see \cite{Nayak-Pal2022}) if for every affine map $T$ and every non-zero constant $\lambda$, $F_p=T\circ F_q \circ T^{-1}$ where $q=\lambda p\circ T$. A quartic polynomial is of the form $p(z)=az^4+bz^3+cz^2+dz+e$
  where $a(\neq 0), b, c, d, e\in \mathbb{C}$. It is well-known that every polynomial $g$ can be transformed to a monic and centered (called normalized) polynomial  by post-composing an affine map $T$ and then multiplying by a suitable non-zero constant $\lambda$. Indeed, by taking $\lambda$ to be the reciprocal of the leading coefficient of $g$ and $T(z)=z+\zeta$ where $\zeta$ is the centroid of $g$ (see page 205, \cite{Beardon_book}), it can be seen that $\lambda g\circ T$ is normalized. As the Chebyshev's method satisfies the Scaling theorem (See Theorem 2.2, \cite{Nayak-Pal2022}), $C_g=T\circ C_{\lambda g\circ T}\circ T^{-1}$. Therefore, without loss of generality we assume that $p$ is normalized, i.e., $a=1$ and $b=0$. Then 
  \begin{equation}\label{qrt_norm}
  p(z)=z^4+cz^2+dz+e.
  \end{equation}
  Though the ongoing discussion is on polynomials, we define the symmetry group of the Julia set of a rational map $R$. It is   denoted by $\Sigma R$ and is defined as $\Sigma R=\{\sigma: \sigma \text{ is an holomorphic Euclidean isometry and }\sigma(\mathcal{J}(R))=\mathcal{J}(R)\}$.
  A normalized polynomial $p$ can be written as 
  \begin{equation}\label{norm_form}
  p(z)=z^\alpha p_0(z^\beta)
  \end{equation}
   where $p_0$ is a monic polynomial, $\alpha \in \mathbb{N}\cup \{0\}$ and $\beta\in \mathbb{N}$ are maximal. We call it normal form of $p$ and it is unique. Then by Theorem 9.5.4,~\cite{Beardon_book},
  $$\Sigma p=\{z\mapsto \lambda z:\lambda^\beta =1\}.$$ Note that  $\beta=1$ if and only if    $\Sigma p$ is trivial. 
\par    If $d=0$ then $p(z)=z^4+cz^2+e$. Further, if $c=0=e$ then $p$ is a monomial, whose Chebyshev's method is a linear map and this case is not of interest. As seen in the following   $\Sigma p$ is non-trivial in all other situations.
  \begin{enumerate}
  	\item[Case 1:]  If $c=0, e\neq 0$ then $p(z)=z^4+e$ and $ \beta=4$.
  	\item[Case 2:] If $c\neq 0, e=0$ then $p(z)=z^2(z^2+c)$ and $ \beta=2$.
  	\item[Case 3:] If $c\neq 0, e\neq 0$ then $p(z)=z^4+cz^2+e$ and $  \beta=2$.
  \end{enumerate}
For the quartic  polynomial $p$ of the form given in Equation \ref{qrt_norm}, let $d \neq 0$. If both $c$ and $e$ are non-zero then $p(z)=z^4+cz^2+dz+e$ is in its normal form with   $\alpha=0, \beta=1$  and $p_0(z)=p(z)$.  If $c \neq 0$ and $e=0$ then  $p$ in its normal form is $z(z^3+cz+d)$, and $\alpha=1=\beta$ and $p_0(z)=z^3+cz+d$. Similarly, if $c=0$ and $e\neq 0$ then $p(z)=z^4+dz+e$  is in  its normal form with $\alpha=0$ and $ \beta=1$.  On the other hand if both $c,e$ are zero then $ p(z)=z(z^3+d) $ and  $\alpha=1, \beta=3$. This is the only case for $d \neq 0$ where $\Sigma p$ is non-trivial.

\begin{enumerate}
	\item[Case 4:] $ p(z)=z(z^3+d) $, $\alpha=1$ and $\beta=3$ in this case.
\end{enumerate}
In this article, we attempt to understand the dynamics of $C_p$ for all quartic $p$ with non-trivial $\Sigma p$. The non-identity elements of $\Sigma p$ are used to understand the dynamics of $C_p$. This approach of determining the dynamics by the symmetries does not work when $\Sigma p$ is trivial. This is a reason why this case needs to be dealt with differently, which we postpone for our future work.

The polynomials in Case $1$ are unicritical. Those in Cases $2$ and $4$ have one of their roots at the origin (the centroid of $p$) and have non-trivial symmetry. The dynamics of  the Chebyshev's method applied to these polynomials and those in Case $3$ having exactly two distinct roots  follows from  Theorem 1.3, \cite{Sym-and-dyn}. In all these cases, the Julia set is found to be connected and $\Sigma p=\Sigma C_p$. Whether, these are true for the rest of the cases is not known.   

\par 
 The polynomials given in Case 3 can not have exactly three distinct roots because zero is not a root and the roots appear in pairs symmetric about the origin.
 Note that, in Case $3$, $p$ can be expressed as $p(z)=(z^2-\gamma_1)(z^2-\gamma_2)$, where $\gamma_1, \gamma_2\in \mathbb{C}\setminus \{0\}$. The Chebyshev's method of $\frac{1}{\gamma_1 ^2}p(\sqrt{\gamma_1}z) $ is conjugate to  $C_p$ by Scaling theorem. Therefore, assume without loss of generality that 
$p(z)=(z^2-1)(z^2-a)$ where $a :=\frac{\gamma_2}{\gamma_1} \in \mathbb{C} \setminus \{0\}$. If $a=-1$ then $p(z)=z^4-1$ is of the form as given in Case $1$; if $a=1$ then $p(z)=(z-1)^2(z+1)^2$ has exactly two roots with the same multiplicity, which is already dealt with in Theorem 1.3 (1)~of  \cite{Sym-and-dyn}. Hence it is sufficient to  consider   $p(z)=(z^2-1)(z^2-a)$ where $a\in \mathbb{C}\setminus\{-1,0,1\}$. Further,  if $|a|>1$ then consider the polynomial $\frac{1}{a^2}p(\sqrt{a}z)=(z^2-\frac{1}{a})(z^2-1)$, whose Chebyshev's method is conjugate to $C_p$ by the Scaling theorem.  
\par 
%For the polynomial $p$ in Equation (\ref{p}), $\Sigma p=\{I, z\mapsto -z\}$. Therefore by Theorem \cite{}, $\Sigma p\subseteq \Sigma C_p$. 
Therefore, it is enough to consider
\begin{equation}\label{p}
p_a(z)=(z^2-1)(z^2-a)
\end{equation}
where $|a| \leq 1 ~\mbox{and}~ a \notin \{-1,0,1\}$. We denote the Chebyshev's method of $p_a$ by $C_a$. 
 This paper takes up the case when $a$ is real. All  extraneous fixed points of $C_a$ are shown to be repelling. As a consequence, it follows that there is no invariant Siegel disk for $C_a$. This is because every such Siegel disk requires  an indifferent fixed point whereas all fixed points are found to be either attracting (when these are the roots of $p_a$) or repelling (where these are extraneous).  It is proved that every periodic Fatou component on which a multiply connected Fatou component lands is an attracting or parabolic domain corresponding to a real periodic point.  Here, we say a Fatou component $U'$ lands on a Fatou component $U$ if there is a $k \geq 0$ such that $C_a ^{k}(U') =U$. In particular this means that there is no Herman ring in the Fatou set of $C_a$. It is also found that the Julia set of $C_a$ is connected if and only if the Julia component containing a non-zero pole is unbounded.
 \par For $a>0,$ we have proved that the immediate basins of $1$ and $-1$ are unbounded and at least two immediate basins (out of four corresponding to the four roots of $p_a$) are simply connected. For $a<0$, all the immediate basins corresponding to the four roots of $p_a$ are found to be unbounded and those corresponding to the purely imaginary roots are shown to be simply connected. Under the assumption that the Fatou set is the union of the basins of attraction of the fixed points corresponding to the roots of $p_a$, it is shown that $\Sigma C_a =\Sigma p_a$ for all positive $a$. For negative $a$ this is proved with an additional assumption that the largest positive extraneous fixed point and the purely imaginary extraneous fixed points are with the same absolute value.
 \par Section 1 describes some useful properties of $C_a$. Results on fixed points and dynamics of $C_a$ are stated and proved  in the first part of Section 2. Then in the two subsections of Section 2, the main results are proved.

By conjugacy, we mean conformal conjugacy throughout this article. 
\section{Basic properties of $C_a$}
For the polynomial $p_a$ in Equation \ref{p},
$$L_{p_{a}}(z)=\frac{6z^6-7(a+1)z^4+(a^2+8a+1)z^2-a(a+1)}{2z^2\{2z^2-(a+1)\}^2},$$
$$1+\frac{1}{2}L_{p_a}(z)=\frac{Q(z)}{4z^2\{2z^2-(a+1)\}^2}$$ where $Q(z)=22z^6-23(a+1)z^4+(5a^2+16a+5)z^2-a(a+1)$,
and $$L_{p_a'}(z)=\frac{12z^2\{2z^2-(a+1)\}}{\{6z^2-(a+1)\}^2}.$$
The Chebyshev's method $C_a$ of $p_a$ is 
\begin{align}
C_a(z)&=z-\left[1+\frac{1}{2}L_{p_a}(z)\right]\frac{p_a(z)}{p'_a(z)}\nonumber\\&=z-\frac{Q(z)p_a(z)}{8z^3\{2z^2-(a+1)\}^3}\label{Chebyshev}\\
&=
\frac{42z^{10}+Az^8+Bz^6+Cz^4+Dz^2+a^2(a+1)}{8z^3\{2z^2-(a+1)\}^3}
\label{Cheby-form}
\end{align}
where $A=-51(a+1)$, $B=4(5a^2+3a+5)$, $C=-3(a^3-7a^2-7a+1)$ and $D=-6a(a^2+3a+1)$.
\par 
	Now, we enumerate  some basic properties of $C_a$. 
\begin{lemma}

%  The Chebyshev's method $C_a$ of $p_a$ satisfies the following properties.
  \begin{enumerate}
  	\item(Degree) The map $C_a$ is an odd rational map of degree ten.  
  	\item(Critical points) The map $C_a$ has eighteen  critical points, counting with multiplicities. The multiple critical points are the four (simple) roots of $p_a$ and three poles of $C_a$. Each of these has multiplicity two as a critical point of $C_a$.  The other four critical points  are simple. If $a$ is real then these four critical points are neither real nor purely imaginary and the poles are all real.
  	\item(Fixed points) The roots of $p_a$ are the superattracting fixed points of $C_a$. The point at infinity is a repelling fixed point of $C_a$ with multiplier $\frac{32}{21}$. The extraneous fixed points are the solutions of $Q(z)=22z^6-23(a+1)z^4+(5a^2+16a+5)z^2-a(a+1)$.
  \end{enumerate}	
\label{basic-cheby}
\end{lemma}
\begin{proof}
	\begin{enumerate}
		\item It follows from Equation~\ref{Cheby-form} that    $C_a$ is an odd rational map with degree ten. 
		\item By Theorem 2.7.1. \cite{Beardon_book}, $C_a$ has eighteen  critical points counting with multiplicity. Recall that the multiplicity of a critical point is one less than the local degree of the function at that point. The derivate of $C_a$ is 
\begin{align}\label{deri}
C_a'(z)&=\frac{(L_{p_a}(z))^2}{2}[3-L_{p'_a}(z)]\nonumber\\
&=\frac{3(z^2-1)^2(z^2-a)^2\{28z^4-8(a+1)z^2+(a+1)^2\}}{8z^4\{2z^2-(a+1)\}^4}
\end{align}
The critical points of $C_a$ are $\pm 1$, $\pm \sqrt{a}$, each of multiplicity $2$ (these are the roots of $p$), $0, \pm \sqrt{\frac{a+1}{2}}$ with multiplicity $2$ each (these are the poles of $C_a$) and the solutions of 
\begin{equation}\label{simple cr}
28z^4-8(a+1)z^2+(a+1)^2=0.
\end{equation} 
The above equation has four distinct roots, namely the solutions of $z^2=\frac{(2\pm i\sqrt{3})(a+1)}{14}$  and each is a simple critical point of $C_a$.
\par That simple critical points are neither real nor purely imaginary and non-zero poles are real whenever $a\in (-1,1)\setminus \{0\}$ is obvious.
\item As the roots of $p_a$ are fixed points as well as critical points of $C_a$, their multiplier is zero. In other words, these are superattracting fixed points of $C_a$. It is evident  from Equation \ref{Cheby-form}  that the degree of the numerator is bigger than that of the denominator of $C_a$ giving that $\infty$ is a fixed point of $C_a$. Its multiplier   is $\frac{32}{21}$ (see page no. 41 \cite{Beardon_book} for the formula). Thus $\infty$ is a repelling fixed point of $C_a$. 
\par The extraneous fixed points of $C_a$ are the solutions of  
\begin{equation}\label{equ_extraneous}
22z^6-23(a+1)z^4+(5a^2+16a+5)z^2-a(a+1)=0
\end{equation} (See Equation~\ref{Chebyshev}).
\end{enumerate}
\end{proof}
\begin{Remark}\begin{enumerate}
	\item(Free critical points) The critical points $\pm 1$ and $\pm \sqrt{a}$ are   superattracting fixed points of $C_a$. The poles $0$ and $\pm \sqrt{\frac{a+1}{2}}$  are also critical points and are mapped to $\infty$ which is a repelling fixed point of $C_a$. Hence the poles are in the Julia set. 
	\par 
	In order to determine the existence of Fatou components different from the basins of superattracting fixed points, the forward orbits of the four simple critical points of $C_a$ need to be followed. These are    $c_1, ~-c_1$ and $c_2, -{c_2}$ where $c_1$ and $c_2$ are principal square roots of $\frac{(2+ i\sqrt{3})(a+1)}{14}$ and $\frac{(2- i\sqrt{3})(a+1)}{14}$ respectively. Following the general practice, we call these as free critical points.
	\item If $c$ is a free critical point then $$C_a(c)=\frac{(42c^8+Ac^6+Bc^4+Cc^2+D)c^2+a^2(a+1)}{8c^3\{2c^2-(a+1)\}^3}.$$
	For $c=c_1,$ 
		 $$C_a(c)=\frac{7^3[(42c^8+Ac^6+Bc^4+Cc^2+D)(2+i\sqrt{3})+14a^2]}{8c(a+1)^3(2+i\sqrt{3})(-5+i\sqrt{3})^3}.$$
	Further, if $a$ is real then  $$C_a(c)=\frac{R(a)+i\sqrt{3}S(a)}{8^3c(a+1)^3(-47+i8\sqrt{3})}$$
	where $R(a)=-(1345a^4+28508a^3+48838a^2+28508a+1345)$ and $S(a)=-3(111a^4-732a^3+3802a^2-732a+111)$. Since	 
	$c_1 =\bar{c_2}$,
	$$C_a(c_2)=-\frac{R(a)-i\sqrt{3}S(a)}{8^3c(a+1)^3(47+i8\sqrt{3})}.$$ Note that if $C_a(c)=0$ then $R(a)=0$ and $S(a)=0$. But $R(a)$ and $S(a)$ can be shown not to have any common root. Hence the critical values corresponding to the free critical points are non-zero whenever $a$ is real.
\end{enumerate}
\end{Remark}
\section{Dynamics of $C_a$ for real parameter}
In view of Equation~\ref{p}, for every real $\lambda$, $C_\lambda$ is conjugate to $C_a$ for some   $a\in (-1,0)\cup (0,1)$.   The study of dynamics of $C_a$ is undertaken in this section. We start with an almost obvious but useful observation. We say a Fatou component $U'$ lands on a Fatou component $U$ if there is a $k \geq 0$ such that $C_a ^{k}(U') =U$. Note that every Fatou component lands on each of its iterated forward image and in particular,  periodic Fatou components land  on themselves. By Sullivan's No Wandering theorem, every Fatou component of a rational map lands on a periodic Fatou component.
\begin{lemma}
	For every non-zero real $a$, $C_a$ preserves the real as well as the imaginary axis. Further, the following are true about the Fatou set of $C_a$.
	\begin{enumerate}
		\item If a Fatou component    intersects  either the real axis  or the imaginary axis then it does not land on any rotation domain (Siegel disk or Herman ring).
		\item No Fatou component  intersects both the real and the imaginary axis.
		\item The Julia component containing the origin is unbounded.
	\end{enumerate} 
	\label{preserve-axes}
\end{lemma} 
\begin{proof}
	 	Since   all the coefficients of  the denominator and numerator  of $ C_a$ are real, $C_a (\mathbb{R}) \subseteq \mathbb{R}$.
	Now for every $y\in \mathbb{R}$, \begin{equation}\label{imm axis}
	C_a(iy)=i\left[\frac{42y^{10}-Ay^8+By^6-Cy^4+Dy^2-a^2(a+1)}{8y^3\{2y^2+(a+1)\}^3}\right] 
	\end{equation}
	where $A,~B,~C$ and $D$ are defined in Equation \ref{Cheby-form}. This shows that the imaginary axis is preserved by $C_a$. 
\begin{enumerate}
\item We prove this assuming that  a Fatou component   intersects  the real axis. The proof in the other case, when the Fatou component intersects  the imaginary axis, is exactly the same.   Let  a Fatou component $U'$ intersect  the real axis and let $r' \in U' \cap \mathbb{R}$.  If $U'$ lands on a rotation  domain $U$ with period $p$,  then  $C_a^{n'}(U')=U$ for some $n' \geq 0$.  Now $C_a^{kp}(r ) \in U $ for all $k$ where $r =C_a^{n'}(r') \in U $.    Since $C_a^{kp}: U \to U$ is conformally conjugate to an irrational rotation $z \mapsto e^{i 2 \pi \theta}z$ (for some irrational number $\theta$)  on a disk or on an annulus about the origin, the set $\{C_a^{kp}(r)\}_{k>0}$ is dense in a Jordan curve contained in $U$. Since $C_a^{kp}(r)$ is real for all $k$, this Jordan curve can only be $\mathbb{R}\cup\{\infty\}$. But $\infty$ is not in the Fatou set leading to a contradiction. 
\item If a Fatou component of $C_a$ intersects both the real and the imaginary axis then the periodic Fatou component $U$  on which it lands  intersects both the axes and hence   can not be a rotation  domain by (1) of this lemma. Therefore, it must be an attracting or parabolic domain. The corresponding periodic points  are in the intersection of both the axes because both the axes are invariant under $C_a$. In other words, the periodic point can only be either $0$ or $\infty$. Since $C_a(0)=\infty$ and $\infty$ is a fixed point, the periodic point must be $\infty$. But $ \infty$ is a repelling fixed point leading to a contradiction.
\item If the Julia component containing the origin is bounded then a Jordan curve can be found in the Fatou set surrounding the origin, i.e., the bounded component of the complement of the Jordan curve contains the origin. The Fatou component containing this Jordan curve has to intersect both the real and the imaginary axes, which is not possible by (2) of this lemma. Therefore, the Julia component containing the origin is unbounded.
\end{enumerate}
\end{proof} 
There are some important consequences. We say a Fatou component surrounds a point if the point is in a bounded component of its complement.
\begin{theorem}
If $U'$ is a multiply connected Fatou component of $C_a$ then there is an $m$ such that $C_a^{m}(U')$ is a Fatou component surrounding a non-zero pole. If $U$ is a periodic Fatou component on which  $C_a^{m}(U')$ lands then $U$ is an attracting or parabolic domain corresponding to a real periodic point. In particular, there is no Herman ring for $C_a$.  
\label{multiply}
\end{theorem}
\begin{proof}
 Consider a Jordan curve $\gamma \subset U'$ such that each of its complementary component intersects the Julia set. Since the Julia set is the closure of the backward orbit of any  point in it (see Theorem 4.2.7 (ii), \cite{Beardon_book}) and all the poles are in the Julia set of $C_a$, there is an $m$ such that $C_a^{m}(\gamma)$ surrounds a pole of $C_a$. Indeed, $m$ can be taken as the smallest natural number such that $C_a^{m}$ is analytic in the bounded component of the complement of $\gamma$ and the preceding statement follows from the Open Mapping Theorem. The curve $C_a^{m}(\gamma)$  can not surround the pole at the origin by Lemma~\ref{preserve-axes}(3) and therefore it surrounds a non-zero pole. The Fatou component $C_a^{m}(U')$ containing  $C_a^{m}(\gamma)$ clearly surrounds the non-zero pole of $C_a$. 
 \par 
	If $U$ is a periodic Fatou component on which  $C_a^{m}(U')$ lands then $U$ intersects the real line because the real line invariant under $C_a$ and $C_a^{m}(U')$ intersects the real line. Since $U$ can not be a rotation domain, it is either an attracting domain or a parabolic domain. Again the invariance of the real line under $C_a$ gives that the  periodic point corresponding  to $U$ must be real. 
	\par   Since Herman rings are multiply connected, it must land on an attracting or parabolic domain which is absurd. Therefore $C_a$ does not have any Herman ring.
	\end{proof}
Note that there are two non-zero poles and the Julia component containing one is bounded if and only if that containing the other is bounded. This is because $C_a$ is an odd function giving that $z\in \mathcal{J}(C_a)$ if and only if $-z\in \mathcal{J}(C_a)$. The boundedness of such Julia components determines the connectedness of the whole Julia set.
\begin{corollary} The Julia set of $C_a $ is connected if and only if the Julia component containing a non-zero pole is unbounded.
\label{connected-JS}
\end{corollary}
\begin{proof}
	 If the  Julia set of $C_a $ is connected then there is only a single Julia component. Since $\infty$ and the non-zero poles are in the Julia set, the  Julia component containing a non-zero pole is unbounded. Conversely, if the Julia set is not connected then there is a multiply connected Fatou component, say $U'$. By Theorem~\ref{multiply}, there is an $m$ such that $C_a^{m}(U')$ surrounds a non-zero pole. In particular, the Julia component containing this non-zero pole is bounded. In other words, if  the Julia component containing a non-zero pole is unbounded then the Julia set is connected.  
\end{proof}

For $0<a<1$, all the roots of $p_a$ are real where as for $-1<a<0$, $p_a$ has two real and two purely imaginary roots. These two cases need to be treated differently. Before we consider them separately, two useful lemmas are presented.
\begin{lemma}[Symmetry about the coordinate axes]
	 The map $z \mapsto -\bar{z}$ preserves the Fatou and the Julia sets of $C_a$ for all real non-zero $a$. In particular, if a Fatou component contains a real (or a purely imaginary) number   then it is symmetric about the real axis (or the imaginary axis respectively).
	 \label{symmetry}
\end{lemma}
\begin{proof}
	 Since $ C_a$ is odd and all the coefficients of its denominator as well as numerator are real, $C_a(-\bar{z}) =-\overline{C_a(z)}$. Therefore, if $\psi(z)=-\bar{z}$ then $C_a^n(z)=\psi^{-1} \circ C_a^n \circ \psi(z)$ for all $n$ and all $z$. This gives that $\psi(\mathcal{F}(C_a))=\mathcal{F}(C_a)$ and  $\psi(\mathcal{J}(C_a))=\mathcal{J}(C_a)$.
\end{proof}
 If $z$ is an extraneous fixed point of $C_a$ then its multiplier is given  by the formula $\lambda(z)=2[3-L_{p_{a}'}(z)]$ (see \cite{Nayak-Pal2022}), which  is nothing but 
 \begin{equation}\label{extra_mult}
 \lambda(z)=2\left[3-\frac{12z^2\{2z^2-(a+1)\}}{\{6z^2-(a+1)\}^2}\right].
 \end{equation}
 \begin{lemma}
 	All the  real extraneous fixed points of $C_a$ are repelling for every non-zero $a \in (-1,1)$.
 	\label{multiplier-extraneous}
 \end{lemma}
 \begin{proof}
 	For every $a \in (-1,1)\setminus \{0\},$ if $ \xi$ denotes the positive square root of $\frac{a+1}{2}$ then $ \lambda(z)$ can be rewritten as $$2\left[3-\frac{6z^2\{z^2- {
 			\xi}^2\}}{\{3z^2- {\xi}^2\}^2}\right].$$
 	In order to determine the multipliers of the real extraneous fixed points, we need to analyse the function
 	 $$ \lambda(z)=2\left[3-\frac{6z^2\{z^2- {\xi}^2\}}{\{3z^2- {\xi}^2\}^2}\right]$$
 	  on the real line. It  suffices to do it in $\{x \in \mathbb{R}: x \geq 0\}$ as the function $z \mapsto \lambda(z)$ is even.
 	
 	The function 	 $\lambda'(x)=-\frac{24 {\xi}^2x(x^2+ {\xi}^2)}{(3x^2- {\xi}^2)^3} >0$ for all $x \in (0, \frac{\xi}{\sqrt{3}})$ and therefore $\lambda(x)$ is increasing in this interval. Consequently, $\lambda(x)>6$ for all $x \in (0, \frac{\xi}{\sqrt{3}}) $. Similarly,  $\lambda'(x)=-\frac{24 {\xi}^2x(x^2+ {\xi}^2)}{(3x^2- {\xi}^2)^3} <0$  for all $x>\frac{ {\xi}}{\sqrt{3}}$. Since  $\lim\limits_{x\to +\infty}\lambda(x)=\frac{14}{3}$, $\lambda(x)$ is a strictly decreasing function in $(\frac{\xi}{\sqrt{3}}, +\infty)$  with minimum value $\frac{14}{3}$. Consequently $\lambda(x)> \frac{14}{3}$ for all real $x$. Thus the multiplier of every real extraneous fixed point of $C_a$ is at least $\frac{14}{3}$. Hence every real extraneous fixed point of $C_a$ is repelling. The graphs of $\lambda(x)$ are given in Figure~\ref{graph-multiplier}.
 \end{proof}
 
\begin{figure}[h!]
	\begin{subfigure}{.52\textwidth}
	\centering
	\includegraphics[width=0.9\linewidth]{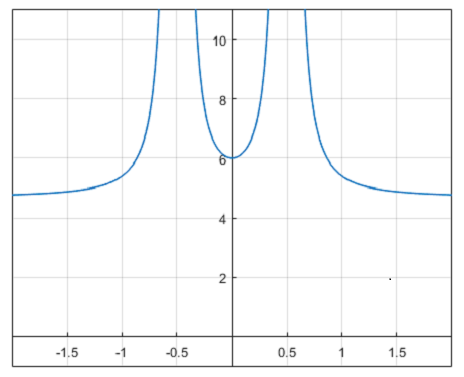}
	\caption{The graph of $\lambda$ for $a=0.5$}
\end{subfigure}
 \hspace{-1.0cm}
	\begin{subfigure}{.52\textwidth}
		\centering
		\includegraphics[width=0.9\linewidth]{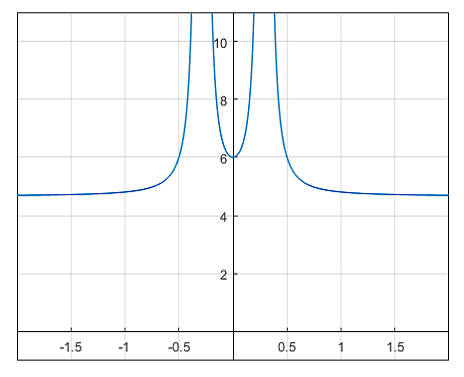}
		\caption{The graph of $\lambda$ for $a=-0.5$}
	\end{subfigure}
	\caption{The graph of $\lambda $ }
	\label{graph-multiplier}
\end{figure}

\subsection{Positive parameter}
 First, we determine the location of  extraneous fixed points. Recall that there are six extraneous fixed point of $C_a$ and these are the solutions of $Q(z)=0$ where 
  $Q(z)$ is as given in   Lemma~\ref{basic-cheby}(3).  
\begin{lemma}
All the extraneous fixed points  of $C_a, 0<a <1$  are in $(-1,1)$ and hence are repelling.
	\label{extraneous} 
\end{lemma}
\begin{proof}
	In view of Lemma!\ref{multiplier-extraneous}, it is enough to show that all the extraneous fixed points of $C_a$ are in $(-1,1)$.
Let $w=z^2$ and $$f(w)=22w^3-23(a+1)w^2+(5a^2+16a+5)w-a(a+1)$$  (see Lemma~\ref{basic-cheby}(3)). Then  $f(0)=-a(a+1)<0$, $f(a)=4a(1-a)^2>0$, $f(\frac{a+1}{2})=-\frac{1}{2}(1-a)^2(1+a)<0$
	and $f(1)=4(1-a)^2>0$. Therefore, $f$ has a root in each of the intervals $(0,a)$, $(a,\frac{a+1}{2})$ and $(\frac{a+1}{2},1)$, and the square roots of these roots are precisely the extraneous fixed points of $ C_a$. If $a_1,a_2,a_3$ are the positive extraneous fixed points in decreasing order then
		$$ 0<a_3<\sqrt{a}<a_2<\sqrt{\frac{a+1}{2}}<a_1<1,$$
		 and the other three extraneous fixed points $-a_1, -a_2, -a_3$ satisfy 	$$-1<-a_1<-\sqrt{\frac{a+1}{2}}<-a_2<-\sqrt{a}<-a_3 <0.$$
		Thus all the extraneous fixed points of $C_a$ for $ 0 <a< 1$ are in $(-1,1)$.
\end{proof}
The graph of $C_{0.5}(x)$, $x\in \mathbb{R}$ is described in Figure \ref{plot C}. Green dots represent the extraneous fixed points, blue dots along with $\pm 1$ are the superattracting fixed points of $C_{0.5}$ corresponding to the roots of $p_{0.5}$, whereas the poles are indicated by the red dots. 
\par 
 From Equation (\ref{deri}), for $x\in \mathbb{R}$, we have 
$$C_a'(x)=\frac{3(x^2-1)^2(x^2-a)^2\{28x^4-8(a+1)x^2+(a+1)^2\}}{8x^4\{2x^2-(a+1)\}^4}.$$ 
As   $28x^4-8(a+1)x^2+(a+1)^2 $ is not zero for any real $x$ and is positive at $x=0$ (see Equation \ref{simple cr}), $C_a'(x)\geq 0$ for all real $x$. Therefore, $C_a$ is increasing in $\mathbb{R}$. Further,
\begin{equation}\label{Comp x}
C_a(x)-x=-\frac{11(x^2-1)(x^2-a)(x^2-a_1^2)(x^2-a_2^2)(x^2-a_3^2)}{32x^3(x^2-\frac{a+1}{2})^3} 
\end{equation}
where $a_1 ,~a_2,~a_3$ are as mentioned in Lemma~\ref{extraneous}. Recall that $C_a$ has four superattracting fixed points, namely $-1,~-\sqrt{a},~\sqrt{a},~1$ and these are all real. Let $\mathcal{A}_{-1},~\mathcal{A}_{-\sqrt{a}},~\mathcal{A}_{\sqrt{a}},~\mathcal{A}_{1}$ be their respective immediate basins of attractions.
\begin{theorem}
	The immediate basins $\mathcal{A}_{  -1}$ and $\mathcal{A}_{  1}$ contain $(-\infty, -a_1)$ and $( a_1, \infty)$ respectively. In particular, these are unbounded and their respective boundaries contain a pole.
	\label{basins-of-one-minusone}
\end{theorem}

\begin{proof} 
It follows from Equation~\ref{Comp x} that for every $x<-1$,  $C_a(x)-x>0$. Since  $C_a : (-\infty,-1) \to (-\infty,-1)$ is strictly increasing, $C_a^n(x)>x$ for every $n\in \mathbb{N}$. This implies that $\{C_a^n(x)\}_{n >0}$ converges to $-1$ for every $x\in (-\infty,-1)$, i.e., $(-\infty,-1]\subset \mathcal{A}_{-1}$.
Now $C_a$ maps $ [-1, -a_1)$ onto itself and $C_a(x)<x$ by Equation~\ref{Comp x}. Since $C_a$ is strictly increasing in this interval, $\lim\limits_{n\to \infty} C_a^n(x)=-1$ for all $x\in [-1, -a_1)$. 
In other words, $(-\infty, -a_1] \subset \mathcal{A}_{-1}$ and in particular $\mathcal{A}_{-1}$ is unbounded.

Since $\mathcal{A}_{-1}$ is invariant and unbounded, it follows from Lemma~4.3 \cite{Nayak-Pal2022} that its boundary  contains a pole.
   \begin{figure}[h!]
	\begin{subfigure}{.6\textwidth}
		\centering
		\includegraphics[width=0.775\linewidth]{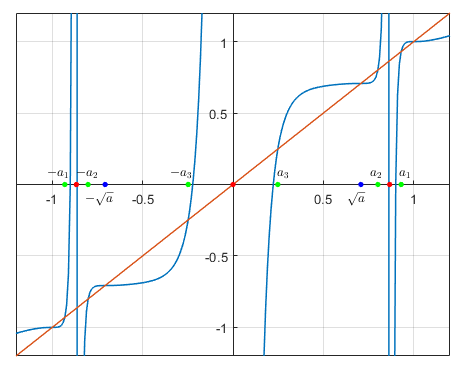}
	\caption{The graph of $C_{0.5}$ }
	\label{plot C}
	\end{subfigure}
	\hspace{-2.38cm}
	\begin{subfigure}{.6\textwidth}
		\centering
		\includegraphics[width=0.775\linewidth]{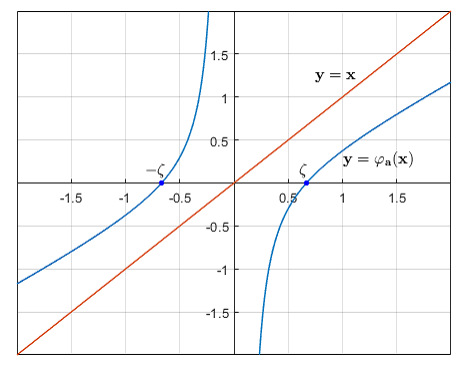}
	\caption{The graph of $\varphi_{- 0.5}$  }
	\label{phi}
	\end{subfigure}
	\caption{The graphs of $C_{a} $ and $\varphi_{a}$ for $a=0.5$ }
	\label{graphs-positive}
\end{figure}
 As $C_a$ is odd, $[1,\infty)\subset \mathcal{A}_1$ and  $\mathcal{A}_1$ is also unbounded. Further, its boundary contains a pole.
\end{proof}
  
\begin{Remark}
	\begin{enumerate}
		\item A similar analysis using Equation~\ref{Comp x} as done in the proof of Theorem\ref{basins-of-one-minusone} gives that $(-a_2, -a_3)\subset \mathcal{A}_{-\sqrt{a}}$ and $(a_3, a_2)\subset \mathcal{A}_{\sqrt{a}}$.
		\item Since $C_a$ is strictly increasing in $(-a_1,- \sqrt{\frac{a+1}{2}}) $ and $C_{a}(-a_1)=-a_1$, the left hand limit of $C_a$ at $- \sqrt{\frac{a+1}{2}} $ is $=+\infty$. Hence $C_a$ has a unique root in  $(-a_1,- \sqrt{\frac{a+1}{2}}) $. Now $(-a_2, -a_3)$ is contained in $\mathcal{A}_{-\sqrt{a}} $ and therefore does not contain any root of $C_a$ (as roots are in the Julia set). Repeating the same argument for $C_a$ in $(-a_3, 0)$, it is found that $C_a$ has a unique  root between  $-a_3$ and $0$. Since $C_a$ is odd, it has two positive  roots, one in  $(0,a_3)$ and the other in  $(\sqrt{\frac{a+1}{2}}, a_1)$. 
	\end{enumerate}
	\label{basins-roots}
	\end{Remark}

   Recall from Equation \ref{imm axis} that $C_a(iy)=i\varphi_{a}(y)$ where $$\varphi_{a}(y)=	  \frac{42y^{10}-Ay^8+By^6-Cy^4+Dy^2-a^2(a+1)}{8y^3\{2y^2+(a+1)\}^3}.$$ Here $A,B,C,D$ are as given in Equation~\ref{Cheby-form}. Further,  $\varphi_{a}$ is a real-valued function defined on the real axis and $\varphi_a=\alpha^{-1}\circ C_a \circ \alpha$, where $\alpha(y)=iy$. Note that for all real non-zero $y$, 
   \begin{equation}\label{deri-imag1}
   \varphi_{a}'(y)=\frac{3(y^2+1)^2(y^2+a)^2\{28y^4+8(a+1)y^2+(a+1)^2\}}{8y^4\{2y^2+(a+1)\}^4}>0.
   \end{equation}
   It is enough to study   the function $\varphi_a$ on the real line  for understanding   $C_a$ on the imaginary axis. 
   
  First we look at the possible zeros of $C_a$ on the imaginary axis.
  \begin{lemma}
  	The function $C_a$ has exactly two purely imaginary roots. 
  \end{lemma}
 
 \begin{proof}
 	Consider $q(x)=42x^{5}-Ax^4+Bx^3-Cx^2+Dx-a^2(a+1)$. Then $q(0)=-a^2(a+1)<0$ and $q(1)=42-A+B-C+D-a^2(a+1)=4\{29+5a(1-a)+a(4-a^2)\}>0$. Since $\lim\limits_{y \to 0^+} \varphi_{a}(y)=-\infty$, there is an $y_0 \in (0,1)$ such that  $\varphi_a(y_0)<0$.
   As $\varphi_a(1)=\frac{q(1)}{8(2+a+1)^3}>0$, $\varphi_a$  has a root in $(0,1)$. This root is unique as $\varphi_a$ is strictly increasing by Equation \ref{deri-imag1}. It follows from the discussion preceding this lemma that $C_a$ has two purely imaginary roots.
 \end{proof}
Here is a remark.
  \begin{Remark}
 Note that $0$ is   a critical point of $\varphi_a$, and $\varphi_a$ is an increasing function on the negative real axis. As $C_a$ has no fixed point on the imaginary axis, the same is true for $\varphi_a$. Since $\lim\limits_{y \to 0^-} \varphi_a(y)=+\infty$, $\varphi_{a}(y)> y$ for all  $y<0$. As $\varphi_a$ is an odd function,  $\varphi_a(y)<y$ for all  $y>0$ (see Figure \ref{phi}).
 \label{varphi-on-imaginary}
  \end{Remark}
Though the imaginary axis does not contain any fixed point of $C_a$, the existence of periodic or pre-periodic points in the imaginary axis can not be ruled out.
\begin{obs}
	The imaginary axis contains two-periodic points of $C_a$ and those are repelling.
\end{obs}  
\begin{proof}
Let $\zeta$ be the positive root of $\varphi_a$ and $I=(0,\zeta)$. Then $\varphi_a(I)=(-\infty,0)$ and $\varphi_a(\varphi_a(I))=\mathbb{R}$. In fact, $\varphi_{a}^2$ maps $I$ bijectively onto $\mathbb{R}$. The branch $g$   of $(\varphi_{a}^2)^{-1}$ such that $g(\mathbb{R})=I$  is a contraction. By the Contraction Mapping Principle, $\varphi_{a}^2$ has a fixed point in $I$. As $\varphi_{a}$ does not have any fixed point on the real line by Remark~\ref{varphi-on-imaginary}, this fixed point of $\varphi_{a}^2$ is a  two periodic point of  $\varphi_a$. Further, this is attracting for $g$ and hence repelling for $\varphi_{a}^2$. Since $\varphi_{a}(y)=-i C_a (iy)$, $C_a$ has a repelling $2$-periodic point on the positive imaginary axis.
\end{proof}
\begin{Remark}
	Using similar arguments, it can be seen that $\varphi_{a}$ has a two periodic point in $(\zeta,+\infty)$. Indeed, this is in the same cycle of the two periodic point mentioned in the above lemma.
\end{Remark}
 As $\varphi_{a}'(y)$ is real for every real $y$, $C_a '(z)$ is a real number for every purely imaginary $z$. Therefore, the periodic points of $C_a$ lying on the imaginary axis can not be irrationally indifferent. Thus, these are attracting, rationally indifferent or repelling. If these are repelling then we  have an important consequence. 
  \begin{lemma}
  	If all the periodic points of $C_a$ lying on the imaginary axis are repelling then the imaginary axis is in the Julia set of $C_a$. 
  \end{lemma}
\begin{proof}
Suppose on the contrary that there is a Fatou component $U'$ intersecting the imaginary axis. Let $U$ be the periodic Fatou component on which  $U' $ lands. Then $U$ intersects the imaginary axis and by Lemma~\ref{preserve-axes}, $U$ can not be a rotation domain.  The other possibility that $U$ is an attracting domain or a parabolic domain would imply that the corresponding attracting or parabolic periodic point must be purely imaginary, which is contrary to the hypothesis. This completes the proof.
\end{proof}

 Though each multiply connected Fatou component is restricted in the sense that it lands on a Fatou component intersecting the real axis, their existence can not be completely ruled out. We are able to show that not all immediate basins of superattracting fixed points of $C_a$ corresponding to the roots of $p_a$ are multiply connected.   
\begin{theorem}
	At least two immediate basins of attraction corresponding to the roots of $p_a$ are simply connected.
	\label{atleast-two-scbasins}
\end{theorem}
\begin{proof}
	If none of the  immediate basins $\mathcal{A}_{ 1}, \mathcal{A}_{ -1}, ~\mathcal{A}_{ \sqrt{a}}, \mathcal{A}_{ -\sqrt{a}}$ contain any   free critical point   then these are simply connected by Theorem 3.9~\cite{Milnor}.
	\par 
	If there is a free critical point say $\eta$  in an immediate basin of attraction $\mathcal{A}_\zeta$ of $\zeta\in \{-1,-\sqrt{a},\sqrt{a},1\}$ then $\bar{\eta} \in  \mathcal{A}_\zeta$. This is because each Fatou component containing a real number is symmetric about the real axis by Lemma~\ref{symmetry}. Further, it follows from the same lemma that $-\eta, -\bar{\eta} \in \mathcal{A}_{-\zeta}$.  Therefore, the other two immediate basins of superattracting fixed points contain no critical points other than the respective roots of $p_a$. Hence these two immediate basins of attraction (corresponding to the roots of $p_a$) are simply connected by Theorem 3.9~\cite{Milnor}.
\end{proof}
Here is a remark.
\begin{Remark}
 Let the boundary of $\mathcal{A}_1$ (or $\mathcal{A}_{-1}$) contain  a non-zero pole. If $\mathcal{A}_1$ is simply connected then   the Julia component containing this non-zero pole is unbounded and it follows from Corollary~\ref{connected-JS} that the Julia set is  connected.  If $\mathcal{A}_1$ is not simply connected then $\mathcal{A}_{\sqrt{a}}$ and $\mathcal{A}_{-\sqrt{a}}$  are simply connected by Theorem~\ref{atleast-two-scbasins}.
\end{Remark}
We have the following result about the symmetry group of the Julia set of $C_a$.
\begin{theorem}
	If the Fatou set of $C_a$ consists only   of the basins of attraction of the superattracting fixed points of $C_a$ then $\Sigma p_a=\Sigma C_a$.
	\label{symmetry-positive}
\end{theorem}
\begin{proof} It is known that $\Sigma p_a \subseteq \Sigma C_a$ and every element of $\Sigma C_a$ is a rotation about the origin, the centroid of $p_a$ (Theorem~1.1.~\cite{Sym-and-dyn}). 
	It  is  shown in Theorem~\ref{basins-of-one-minusone} that the immediate basins $\mathcal{A}_{\pm 1}$ are unbounded. As $\infty$ is a repelling fixed point, by Lemma 3.2. \cite{Sym-and-dyn}, every Fatou component landing on these immediate basins $\mathcal{A}_{\pm 1}$ is bounded. If $\mathcal{A}_{\pm \sqrt{a}}$ are bounded then every Fatou component landing on these will be bounded as $\infty$ is a fixed point. On the other hand, if $\mathcal{A}_{\pm \sqrt{a}}$ are unbounded then Lemma 3.2. \cite{Sym-and-dyn} gives  that every Fatou component landing on $\mathcal{A}_{\pm \sqrt{a}}$ is bounded.
	Hence the Fatou set of $C_a$ contains at most four unbounded components.
	
	Let $\sigma \in \Sigma C_a$. Then $\sigma(\mathcal{A}_1)$ can not be equal to $\mathcal{A}_{-\sqrt{a}}$ or $\mathcal{A}_{\sqrt{a}}$ and therefore  $\sigma(\mathcal{A}_{1})=\mathcal{A}_{1}$ or $ \mathcal{A}_{-1}$. Thus $\sigma$ is either the identity or $z\mapsto -z$. Since $z \mapsto -z$ is the only non-identity element of $\Sigma p_a$, we have  $\sigma \in  \Sigma p_a $.
	 Thus $ \Sigma C_a \subseteq \Sigma p_a$ and hence $ \Sigma C_a = \Sigma p_a$. 
\end{proof}

  \begin{Remark}
  	
  \begin{enumerate}
  	\item If $a>0$ and  the Fatou set of $C_a$ consists of only the basins of the fixed points of $C_a$ corresponding to the roots of $p_a$ (all of these are real) then there is no Fatou component intersecting the  imaginary axis. This is because the imaginary axis is invariant and   no Fatou component can intersect both the axes. Therefore, the imaginary axis is in the Julia set of $C_a$ where $a>0$ and $\mathcal{F}(C_a)$ consists of only the basins of the fixed points of $C_a$ corresponding to the roots of $p_a$.
  	\item It is not known whether the non-zero poles are on the boundary of $\mathcal{A}_{\pm 1}$ or not.
  	 
  	\item For $ a>0,$ it is believed (but yet not proved) that the immediate basins $\mathcal{A}_{\pm \sqrt{a}}$ are unbounded. This is supported by Figure \ref{Julia set}, which is generated using MatLab.
  \end{enumerate}
\end{Remark}
\begin{figure}[h!]
	\begin{subfigure}{.6\textwidth}
		\centering
		\includegraphics[width=0.75\linewidth]{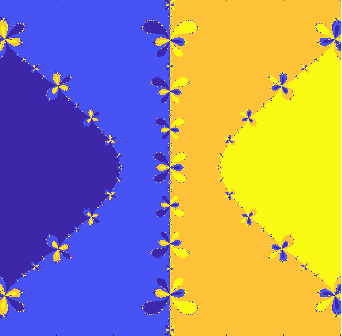}
		\caption{ $a=0.5$}
		\label{Julia set}
	\end{subfigure}
	\hspace{-2.0cm}
	\begin{subfigure}{.6\textwidth}
		\centering
		\includegraphics[width=0.75\linewidth]{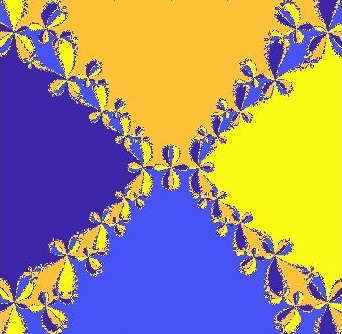}
		\caption{ $a=-0.5$}
		\label{Quartic_b}
	\end{subfigure}
	\caption{The Julia sets of $C_a $ }
	\label{JS}
\end{figure}
Figure \ref{Julia set} illustrates the Fatou and the Julia sets of $C_{0.5}$. The largest regions in deep blue, blue, yellow and deep yellow represent the immediate basins of attractions of $-1, -\sqrt{a},  \sqrt{a}$ and $1$ respectively. All the  smaller regions in deep blue belong to the basin (but not the immediate basin) of $-1$. Similar is the case of smaller regions in other three colours. The Julia set is the complement of the union of these four basins.  

\subsection{  Negative parameters}
We  are to deal with  $p_a (z)=(z^2 -1)(z^2-a)$ for $-1< a<0$.
Let $a=-b$ where $b\in (0,1)$ so that  $$p_{-b}(z)=(z^2-1)(z^2+b)$$ for $0<b<1$. Then the Chebyshev's method of of $p_{-b}$, denoted by $C_{-b}$, is defined as
\begin{equation}\label{Chebyshev2}
C_{-b}(z)=
\frac{42z^{10}+\tilde{A}z^8+\tilde{B}z^6+\tilde{C}z^4+\tilde{D}z^2+b^2(1-b)}{8z^3\{2z^2-(1-b)\}^3}
 \end{equation}
where $\tilde{A}=-51(1-b)$, $\tilde{B}=4(5b^2-3b+5)$, $\tilde{C}= 3(b^3+7b^2-7b-1)$ and $\tilde{D}=6b(b^2-3b+1)$. The  critical points of $C_{-b}$ are $\pm 1, \pm i\sqrt{b}, 0, \pm \sqrt{\frac{1-b}{2}}$, each with multiplicity two and the solutions of $28z^4-8(1-b)z^2+(1-b)^2=0$ (see Equation \ref{simple cr}) which are all simple.
If $z$ is such a solution then $z^2=\frac{2\pm i\sqrt{3}}{14}(1-b)$. 

\par  
We prove as in positive parameter case that all extraneous fixed points are repelling.
\begin{lemma}
	All the extraneous fixed points of $C_{a}$ for $-1< a<0$ are repelling.
	\label{extraneous-negative-parameter}
\end{lemma}
\begin{proof}
	The extraneous fixed points are the solutions of Equation \ref{equ_extraneous}. Let $w=z^2$ and $f(w)=22w^3-23(1-b)w^2+(5b^2-16b+5)w+b(1-b)$. Then $f(-b)= -4b(1+b)^2<0$, $f(0)=b(1-b)>0$, $f(\frac{1-b}{2})=-\frac{1}{2}(1-b)(1+b)^2<0$
	and $f(1)=4(1+b)^2>0$.	Therefore $f$ has a root in each of the intervals $(-b,0),~(0,\frac{1-b}{2})$ and $(\frac{1-b}{2},1)$ and the square roots of these roots are precisely the extraneous fixed points of $C_{-b}$. There are four real and two purely imaginary  extraneous fixed points. If $b_1, b_2$ are the positive extraneous fixed points  in decreasing order then 
	$$-1<-b_1<-\sqrt{\frac{1-b}{2}}<-b_2<0<b_2<\sqrt{\frac{1-b}{2}}<b_1<1.$$
	It follows from Lemma~\ref{multiplier-extraneous} that these four (real) extraneous fixed points are repelling. 
	\par  
	If $ib_3, -ib_3$ are purely imaginary extraneous fixed points such that $b_3 >0$ then $ib_3$ is a square root of the negative root of $f$ (lying in $(-b,0)$) and therefore, 
	$$-\sqrt{b}<-b_3<0<b_3<\sqrt{b}.$$
	\par
%  As $\lambda$ is continuous and deceasing in $(\frac{\tilde{\xi}}{\sqrt{3}}, \infty)$, $\lambda(x)>>1$. As $\frac{\tilde{\xi}}{\sqrt{3}}<\tilde{\xi}$ and $\lambda(-x)=\lambda(x)$ for any $x\in \mathbb{R}$, $\lambda(x)>1$ for every $x\in \mathbb{R}$.
%
% Consider the real valued function $\lambda$, defined on real line by
%\begin{equation}
%\lambda(x)=2\left[3-\frac{6x^2\{x^2-\tilde{\xi}^2\}}{\{3x^2-\tilde{\xi}^2\}^2}\right].
%\end{equation}
%Then for any $x\in (-\tilde{\xi},\tilde{\xi})$, $\frac{6x^2\{x^2-\tilde{\xi}^2\}}{\{3x^2-\tilde{\xi}^2\}^2}<0$, hence $\lambda(x)>6$.
%\par 
%\par 
%For a real extraneous fixed point $x_0$ of $C_{-b}$, the multiplier is $\lambda(x_0)>1$. Thus all real extraneous fixed points of $C_{-b}$ are repelling.
\par   If $z=iy$ is an extraneous fixed point of $C_{-b}$ then its multiplier is given by  $2\left[3-\frac{6y^2\{y^2+\tilde{\xi}^2\}}{\{3y^2+\tilde{\xi}^2\}^2}\right]$ where $\tilde{\xi} =\sqrt{\frac{1-b}{2}}>0$. Let $\tilde{\lambda}:\mathbb{R}\to \mathbb{R}$ be    defined by $$\tilde{\lambda}(y)=2\left[3-\frac{6y^2\{y^2+\tilde{\xi}^2\}}{\{3y^2+\tilde{\xi}^2\}^2}\right].$$
Then 
\begin{equation} \tilde{\lambda}^{'}(y)=\frac{24\tilde{\xi}^2 y(y^2-\tilde{\xi}^2)}{(3y^2+\tilde{\xi}^2)^3}.
\label{derivative-lambda-tilde}
\end{equation}
Since $\tilde{\lambda}$ is even, it is enough to analyse it in $[0, +\infty)$. Since $C_{-b}^{'}(ib_3)=\tilde{\lambda}(b_3)$,  $ib_3$ is a repelling fixed point of $C_{-b}$ if and only if $|\tilde{\lambda}(b_3)|>1 $. We are going to establish this by showing that $\tilde{\lambda}(y)>1$ for all $y\in (0, \sqrt{b})$. This is so because $0 < b_3 <\sqrt{b}$. There are two cases depending on whether  $0< b \leq  \frac{1}{3}$ or  $\frac{1}{3} < b <1  $.

%e critical points of $\tilde{\lambda}$ are $0$ and $  \tilde{\xi}$.
\par 
If $0< b \leq  \frac{1}{3}$ then  $\tilde{\xi}^2 \geq b$ and  $\tilde{\lambda}'(y)<0$  for all $y\in (0, \sqrt{b})$. Therefore $\tilde{\lambda}$ is a decreasing function in $(0,\sqrt{b})$ and consequently, $\tilde{\lambda}(y)\geq \tilde{\lambda}(\sqrt{b})=6\left[\frac{21b^2+6b+1}{(5b+1)^2}\right].$ Letting $\tilde{s}(b)= \frac{21b^2+6b+1}{(5b+1)^2}$, it is seen that    $\tilde{s}^{'}(b)=\frac{12b -4}{(5b+1)^3}<0$ in   $(0,\frac{1}{3})$ and its minimum   value is attained at $b=\frac{1}{3}$. Therefore $\tilde{s}(b)\geq \tilde{s}(\frac{1}{3})=\frac{3}{4 }$ for all $0<b \leq \frac{1}{3}$. This gives that $\tilde{\lambda}(y)\geq \frac{9}{2}>1$ for all $y \in (0,\sqrt{b})$. 
\par 
If $\frac{1}{3} < b <1  $ then  $\tilde{\xi}^2<b$  and $\tilde{\lambda}$  has a critical point $\tilde{\xi}$ in the interval $(0,\sqrt{b})$ (see Equation~\ref{derivative-lambda-tilde}). Indeed, $\tilde{\lambda}$ decreases in $(0, \tilde{\xi})$ and then increases in $(\tilde{\xi}, \sqrt{b} )$ attaining its minimum at $\tilde{\xi}$. Therefore  $\tilde{\lambda}(y)>\tilde{\xi}=\frac{9}{2}>1$ for any $y\in (0,\sqrt{b})$.
\par 
  This concludes the proof.
\end{proof}
Unlike the case of positive parameter, all the superattracting basins corresponding to the roots of $p_a$ are found to be unbounded in this case.
\begin{theorem}\label{unbdd_imm}
All immediate basins corresponding to the superattracting fixed points of $C_{-b}$ are unbounded.
\end{theorem}
\begin{proof}
	Recall that the superattracting fixed points of $C_{-b}$ are $-1, 1, i \sqrt{b}$ and $-i \sqrt{b}$.   In view of Lemma~\ref{symmetry}, it is enough to prove that the immediate basins $\mathcal{A}_{1}$ and $\mathcal{A}_{ i \sqrt{b}} $ corresponding to $1$ and $ i \sqrt{b}$ respectively are unbounded.
\par 
To show that  $\mathcal{A}_{1}$ is unbounded, we need to analyse the iterative behavour of $C_{-b}$ on $\mathbb{R}$.  For $x\in \mathbb{R}$,  
$$C_{-b}'(x)=\frac{3(x^2-1)^2(x^2+b)^2\{28x^4-8(1-b)x^2+(1-b)^2\}}{128x^4\{x^2-\tilde{\xi}^2\}^4},$$ where $\tilde{\xi}$ is the positive square root of $ \frac{1-b}{2}$.
Since all simple critical points of $C_{-b}$ are non-real, $C_{-b}'(x)>0$ for every $x\in \mathbb{R}\setminus \{\pm 1,0,\pm \tilde{\xi}\}$. In particular, $C_{-b}$ is increasing in  $[1,\infty)$.

Also
$$C_{-b}(x)-x=-\frac{11(x^2-1)(x^2+b)(x^2-b_1^2)(x^2-b_2^2)(x^2+b_3^2)}{32x^3(x^2-\tilde{\xi}^2)^3}$$ where $\pm b_1,\pm b_2$ are real and $\pm i b_3$ are purely imaginary extraneous fixed points of $C_{-b}$ (see Lemma~\ref{extraneous-negative-parameter}). It follows from Lemma~\ref{extraneous-negative-parameter} that  $b_1^2, b_2^2, \tilde{\xi}^2<1$ which gives that 
 $C_{-b}(x)<x$ for all $x \in [1,\infty)$.
Now $\{C_{-b}^n(x)\}_{n>0}$ is a decreasing sequence which is bounded below by $1$ for each   $x\in [1, \infty)$. Therefore, $\lim\limits_{n \to \infty}C_{-b}^n(x) = 1$ and hence $[1,\infty)\subseteq \mathcal{A}_{1}$. Figure~\ref{plot C_b} illustrates the case when $b=0.5$.  
  \begin{figure}[h!]
	\begin{subfigure}{.62\textwidth}
		\centering
		\includegraphics[width=0.75\linewidth]{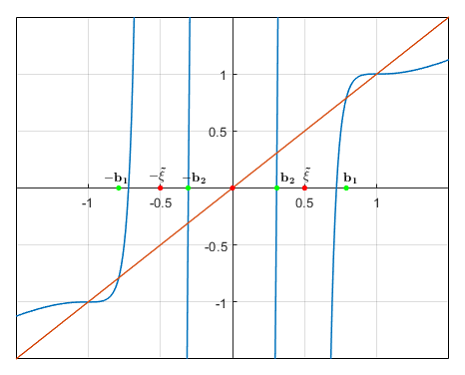}
	\caption{The graph of $C_{- 0.5}$}
	\label{plot C_b}
	\end{subfigure}
	\hspace{-2.250cm}
	\begin{subfigure}{.62\textwidth}
		\centering
		\includegraphics[width=0.75\linewidth]{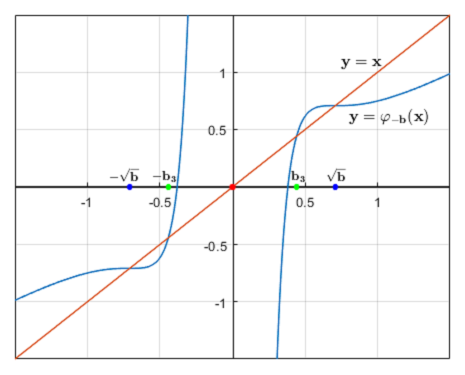}
	\caption{The graph of $\varphi_{-0.5}$}
	\label{-phi}
	\end{subfigure}
	\caption{The graphs of $C_{a} $ and $\varphi_{a}$ for $a=-0.5$}
%	\label{graphs-positive}
\end{figure}
\par      
To show that  $\mathcal{A}_{ i \sqrt{b}} $ is unbounded, first note that $C_{-b}(iy)=i\varphi_{-b}(y)$ for $y\in \mathbb{R}$ where $$\varphi_{-b}(y)=  
\frac{42y^{10}-\tilde{A}y^8+\tilde{B}y^6-\tilde{C}y^4+\tilde{D}y^2-b^2(1-b)}{8 y^3\{2y^2+(1-b)\}^3}
,$$
where $\tilde{A}=-51(1-b)$, $\tilde{B}=4(5b^2-3b+5)$, $\tilde{C}= 3(b^3+7b^2-7b-1)$ and $\tilde{D}=6b(b^2-3b+1)$.
This follows from  Equation \ref{Chebyshev2}. The dynamics of $C_{-b}$ on the imaginary axis is the  same as that of $\varphi_{-b}$ on the real line. The unboundedness of $\mathcal{A}_{\sqrt{b}}$ will be proved by showing that for each $y \in (\sqrt{b}, \infty)$, $\lim\limits_{n \to \infty}\varphi_{-b}^{n}(y)=\sqrt{b}$.

Observe that
$$\varphi_{-b}'(y)=\frac{3(y^2+1)^2(y^2-b)^2\{28y^4+8(1-b)y^2+(1-b)^2\}}{128y^4\{y^2+\tilde{\xi}^2\}^4}.$$ The equation $28y^4+8(1-b)y^2+(1-b)^2=0$ has no real root ( else $y^2$ will be equal to $\frac{-2\pm i\sqrt{3}}{14}(1-b) $ which is not possible). Therefore $\varphi_{-b}'(y)>0$ for every $y\in \mathbb{R}\setminus \{0,\pm \sqrt{b}\}$. Since 
$$\varphi_{-b}(y)-y=-\frac{11(y^2+1)(y^2-b)(y^2+b_1^2)(y^2+b_2^2)(y^2-b_3^2)}{32y^3(y^2+\tilde{\xi}^2)^3}<0,$$ $\varphi_{-b}(y)<y$ for all $y >\sqrt{b}$. Therefore $\{\varphi_{-b}^{n}(y)\}_{n>0}$ is a decreasing sequence which is bounded below by $\sqrt{b}$ and hence $\lim\limits_{n \to \infty}\varphi_{-b}^{n}(y)=\sqrt{b}$ for all $y >\sqrt{b}$ (see Figure \ref{-phi}). 
\end{proof}
\begin{Remark}\label{comp_imm}
Following a similar argument used in the proof of Theorem~\ref{unbdd_imm}, it can also be shown that $\lim\limits_{n \to \infty}C_{-b}^n(x) = 1$ whenever $x\in(b_1,1)$, where $b_1$ is the extraneous fixed point of $C_{-b}$ lying on $(\tilde{\xi},1)$, and $\lim\limits_{n \to \infty}\varphi_{-b}^{n}(y)=\sqrt{b}$ for all $y\in (b_3,\sqrt{b})$, where $b_3$ is the purely imaginary extraneous fixed point of $C_{-b}$ lying on $(0,i\sqrt{b})$. Thus $(b_1, \infty)\subset \mathcal{A}_1$ and $(ib_3, \infty)\subset \mathcal{A}_{i\sqrt{b}}$, where  $(ib_3, \infty)$ is an interval on the imaginary axis. 
\end{Remark}
The next theorem assures the simply connectedness of the immediate basins of the purely imaginary superattracting fixed points of $C_{-b}$. The connectedness of the Julia set is also proved under a condition.
\begin{theorem}
	The immediate basins $\mathcal{A}_{ i \sqrt{b}}$ and $\mathcal{A}_{- i \sqrt{b}}$ are simply connected.
	If there is a non-zero pole on the boundary of any of these immediate basins then the Julia set of $C_{-b}$ is connected.
\end{theorem}
\begin{proof}
	In view of Lemma~\ref{symmetry}, it is enough to prove this theorem for  $\mathcal{A}_{ i \sqrt{b}}$.
The immediate basin  $\mathcal{A}_{ i \sqrt{b}}$ does not intersect the real line by Lemma~\ref{imm axis} and all the poles of $C_{-b}$ are real. It follows from the arguments used in the proof of  Theorem~\ref{multiply} that it is simply connected.
\par The Julia component containing the origin is unbounded by Lemma~\ref{preserve-axes}(3). If there is a non-zero pole on the boundary of  $\mathcal{A}_{ i \sqrt{b}}$ then this non-zero pole is also in the unbounded Julia component. We are now done by Corollary~\ref{connected-JS}. 
\end{proof}
That the symmetry groups of $p_{-b}$ and $C_{-b}$ coincide in some case is now proved.
\begin{theorem}
	If the Fatou set of $C_{-b}$ consists only   of the basins of attraction of the superattracting fixed points of $C_{-b}$ and $b_1 \neq b_3$ where $b_1$ is the largest positive extraneous fixed point  $C_{-b}$ and $ib_3$ is the extraneous fixed point of $C_{-b}$ lying on the imaginary axis  then $\Sigma p_{-b}=\Sigma C_{-b}$.
\end{theorem}
\begin{proof}
	First note that $\Sigma p_{-b} \subseteq \Sigma C_{-b}$ and every element of $\Sigma C_{-b}$ is a rotation about the origin, the centroid of $p_{-b}$ (Theorem~1.1.~\cite{Sym-and-dyn}). 
Following the proof of Theorem~\ref{symmetry-positive}, we get that there are exactly four unbounded components in $\mathcal{F}(C_{-b})$. Therefore, for any $\sigma\in \Sigma C_{-b}$, $\sigma(\mathcal{A}_1)$ is either $\mathcal{A}_{\pm i\sqrt{b}}$ or $\mathcal{A}_{-1}$. From Remark \ref{comp_imm},  we get $(b_1,\infty)\subset \mathcal{A}_1$, whereas, the interval $(ib_3,\infty)$ on the imaginary axis is in $\mathcal{A}_{i\sqrt{b}}$. Thus if $\sigma(\mathcal{A}_{1})=\mathcal{A}_{i\sqrt{b}}$ then $\sigma((b_1,\infty))=(ib_3,\infty)$. As the extraneous fixed points $b_1$ and $ib_3$ are in the Julia set, this can only possible whenever $b_1=b_3$, that contradicts our assumption. Therefore $\sigma(\mathcal{A}_1)\neq \mathcal{A}_{i\sqrt{b}}$. Since two purely imaginary extraneous fixed points are with same modulus, by the similar argument we get $\sigma(\mathcal{A}_1)\neq \mathcal{A}_{-i\sqrt{b}}$. Therefore $\sigma(\mathcal{A}_1)=\mathcal{A}_{-1}$. As $\sigma$ is an arbitrary element in $\Sigma C_{-b}$, we get $\Sigma C_{-b}=\{I,z\mapsto -z\}$.  Since $z \mapsto -z$ is the only non-identity element of $\Sigma p_{-b}$,  $\Sigma C_{-b}  \subseteq \Sigma p_{-b}$.
\end{proof}
The Fatou set of $C_{-0.5}$ is given in Figure~\ref{Quartic_b}. The regions with deep blue, blue, yellow and deep yellow signify the basins of attraction of the four super attracting fixed points of $C_{-0.5}$. The largest region of each colour is the respective immediate basin. The Julia set of $C_{-0.5}$ is the complement of the union of these four basins. Lastly, we provide the following table illustrating some comparisons between the two cases: $a<0$ and $a>0$.
 \begin{table}[h!]
 	\centering
 	\begin{tabular}{c|c}\hline\hline
 		\begin{minipage}{7cm}
 			\centering $a>0$
 		\end{minipage}&\begin{minipage}{7cm}
 			\centering $a<0$
 		\end{minipage}\\\hline\hline
 		\multicolumn{2}{c}{\begin{minipage}{14.5cm}
 				\vspace{0.1cm}The real and imaginary axes are invariant and $\mathcal{J}(C_a)$ is symmetric with respect to both the axes.
 		\end{minipage}}\vspace{0.1cm}\\\hline
 		\multicolumn{2}{c}{\begin{minipage}{14.5cm}
 				\vspace{0.15cm}All roots of $p_a$ and poles of $C_a$ are critical points of $C_a$ with multiplicity two each. There are four other simple critical points of the form $c$, $-c$, $\bar{c}$ and $-\bar{c}$, where $c^2=\frac{(2+i\sqrt{3})(a+1)}{14}$. Thus simple critical points are non-real.
 		\end{minipage}}\vspace{0.1cm}\\\hline\multicolumn{2}{c}{There are six extraneous fixed points.}\\
 		\hline 
 		\begin{minipage}{7cm}
 			All extraneous fixed points are real and repelling.
 		\end{minipage}&\begin{minipage}{7cm}
 			\vspace{0.1cm} Four extraneous fixed points are real and two are purely imaginary. All are repelling.\vspace{0.1cm}
 		\end{minipage}\\
 		\hline
 		\begin{minipage}{7cm}
 			\vspace{0.1cm}Immediate basins of $1$ and $-1$ are unbounded.\vspace{0.1cm}
 		\end{minipage} & \begin{minipage}{7cm}
 			Immediate basins $\mathcal{A}_{\pm 1}$ and $\mathcal{A}_{\pm \sqrt{a}}$ are unbounded.
 		\end{minipage}\\
 		\hline
 		\begin{minipage}{7cm}
 			At least two immediate basins are simply connected.
 		\end{minipage}&\begin{minipage}{7cm}
 			\vspace{0.1cm}At least two immediate basins are simply connected. More precisely, $\mathcal{A}_{\pm \sqrt{a}}$ are simply connected.\vspace{0.1cm}
 		\end{minipage}\\
 		\hline
 		\multicolumn{2}{c}{There is no Herman ring.}\\
 		\hline
 		\multicolumn{2}{c}{There is no invariant Siegel disk.}\\\hline
 	\end{tabular}
 	\caption{Properties of $C_a$}
 	\label{Table}
 \end{table}
\clearpage
\section{Declarations}

\subsection{Funding} The second author is supported by the University Grants Commission, Govt. of India.

\subsection{Conflicts of interest/Competing interests}
Not Applicable.
\subsection{Data Availability statement} Data sharing not applicable to this article as no datasets were generated or analysed during the current study.

\subsection{
	Code availability } Not Applicable

\end{document}